\documentclass[12pt]{article}

\usepackage[margin=1in]{geometry}  
\usepackage{amsmath}               
\usepackage{amsfonts,amssymb}              
\usepackage{amsthm}                
\usepackage{color}
\usepackage{enumerate}
\usepackage{bbm}
\usepackage{graphicx}

\newtheorem{theorem}{Theorem}[section]
\newtheorem{lemma}[theorem]{Lemma}
\newtheorem{proposition}[theorem]{Proposition}
\newtheorem{corollary}[theorem]{Corollary}
\newtheorem{definition}{Definition}
\theoremstyle{remark}
\newtheorem{example}[theorem]{Example}
\newtheorem{remark}{Remark}

\newcommand{\R}{\mathbb R}

\newcommand{\E}{\mathbf {E}}

\newcommand{\pr}{\mathbb{P}}

\DeclareMathOperator{\dist}{dist}

\providecommand{\keywords}[1]{\textit{Keywords:} Mixed Fractional Brownian motion, Entropy minimization, Fredholm integral equation, weakly singular kernel, numerical solutions#1}
\providecommand{\msc}[1]{\textit{AMS MSC 2010:} #1}

\allowdisplaybreaks
\begin{document}

\title{Approximate solution of the integral equations involving kernel with additional singularity}

\author{Vitalii  Makogin \\
Institute of Stochastics, Ulm University,  D-89069 Ulm, Germany. \\
E-mail: vitalii.makogin@uni-ulm.de
\and
 Yuliya Mishura \\
Department of Probability Theory, Statistics and Actuarial
Mathematics \\
Taras Shevchenko National University of Kyiv \\
Volodymyrska 64, Kyiv 01601, Ukraine, E-mail:myus@univ.kiev.ua
\and
 Hanna Zhelezniak\\
Taras Shevchenko National University of Kyiv,\\
Volodymyrska 64, Kyiv 01601, Ukraine,
E-mail:hanna.zhelezniak@gmail.com
}
\maketitle
\pagestyle {myheadings}
\markright{Approximate solution of the integral equations}

\begin{abstract}
The   paper is devoted to the approximate solutions of the Fredholm integral equations of the second kind with the weak singular kernel that can have additional singularity in the numerator. We describe two problems that lead to such equations. They are the problem of minimization of small deviation and the entropy minimization problem. Both of them appear when considering dynamical system involving mixed fractional Brownian motion. In order to deal with the kernel with additional singularity applying well-known methods for weakly singular kernels, we
prove  the theorem on the approximation
of solution of integral equation with the kernel containing additional singularity by
the solutions of the integral equations whose kernels are weakly singular but the numerator is
continuous. We demonstrate numerically how our methods work being applied to our specific integral equations. 
\end{abstract}

\keywords{.}

\msc{60G22, 45L05, 45B05, 34K28, 26A33.}
\section{Introduction}\label{sec1}
The present paper is devoted to the approximate solutions of the Fredholm integral equations of the second kind  on the interval $[0, T]$, with the  kernel of the form $K(t,s)=L(t,s)|t-s|^{-\alpha}, \alpha\in(0,1),$ where the numerator $L(t,s), (t,s)\in[0,T]^2$ is bounded and continuous a.e. with respect to the Lebesgue measure but can have the points of discontinuity on $[0,T]^2$, due to the need to approximately solve such equations when considering some optimization problems associated with mixed Brownian-fractional Brownian motion. To the best of our knowledge, in the numerous papers and books devoted to  this topic, the kernel $L$ is continuous. Such kernels are called weakly singular. In no way claiming   completeness of the  bibliographic references,  we only mention in this connection the classic textbook \cite{kant}   which we find very useful when considering integral equations with singular kernels. As for approximate methods for solving integral equations, we mention the monographs \cite{Atk1976, MR1464941, Baker78, Delves85, Prem} and papers \cite{BabArz, Neta, zeng2019}, which show various approximation methods, but both in these   and in other works, the numerator is assumed to be at least continuous, and often differentiable.

However,  we are faced with real problems whose process of solving led to Fredholm equations with a weakly singular kernel, the numerator of which is not a continuous function. We called such a kernel as having an additional singularity. 
We present the problems which lead to the integral equations involving the kernels with
additional singularity. These problems were discussed in detail in the papers \cite{Melnikov, MishMakog, Makmish}. More precisely, the paper \cite{Melnikov} is devoted to the problem of optimization of
small deviation for mixed fractional Brownian motion with trend for the case of the Hurst index $H\in(0; 1/2)$;
whereas in the paper  \cite{MishMakog}, the same problem was
considered for the case $H\in(1/2; 1)$. The paper \cite{ Makmish} is devoted to the minimization of entropy in the system described by the mixed fractional Brownian motion with trend. Both cases, $H\in(0; 1/2)$ and $H\in(1/2; 1)$ were considered and as the result, the problem was reduced to the couple of the same integral equations as in the case where the minimization of small deviations was studied. In the case $H\in(0; 1/2)$ the integral equation contains the kernel with additional singularity whereas in the case   $H\in(1/2; 1)$ the kernel is simply weakly singular.  In his connection, in order to deal with the kernel with additional singularity applying well-known methods for weakly singular kernels, we 
prove  the theorem on the approximation
of solution of integral equation with the kernel containing additional singularity by
the solutions of the integral equations whose kernels are weakly singular but the numerator is
continuous. We demonstrate numerically how our methods work being applied to our specific integral equations. 

The paper is organized as follows.   In Section \ref{sec2}, we present two problems which lead to the integral equations involving the kernels with additional singularity. Roughly speaking, they are the problem of minimization of small deviation and the entropy minimization problem. Both of them appear when considering dynamical system involving mixed fractional Brownian motion. In Subsection \ref{subsec1} we describe the problems themselves and then, in Subsection  \ref{subsec2}, we explain how to reduce these optimization problems  to the integral
equation and describe the structure of the involved integral kernels. The representations of the kernels are   new, in comparison with the papers \cite{Melnikov, MishMakog, Makmish}, and they are much more convenient for the numerical solution.   In Section \ref{sec3}, we provide approximation of the involved kernels  $K(t,s)$ by kernels with continuous numerators. The main result of this section  is  the Theorem \ref{ContKern} on the approximation of solution of integral equation with the kernel containing additional singularity by   the solutions of the integral equations whose kernels are weakly singular.
 Section \ref{sec4} is devoted to a numerical solution of the considered Fredholm integral equations.  We describe the  modification of product-integration method of the numerical solution in Subsection \ref{subsec11}. We illustrate it by numerical experiments, the graphs of corresponding kernels and solutions, provide a short sensitivity study of errors in Subsection \ref{subsec12}.

\section{Integral equation  appearing in the problem of optimization of small deviation and entropy functionals}\label{sec2}
Here, we present the problems which lead to the integral equations involving the kernels with additional singularity.
\subsection{Description of the problems of the optimization of small deviation and the entropy minimization problem}\label{subsec1} Let  $(\Omega, \mathcal{F}, (\mathcal{F}_t)_{t \in [0, T]}, \pr)$ be a filtered probability space that supports all the stochastic  processes presented below, and it is assumed that they are all adapted to this filtration. Now, introduce two independent stochastic processes, namely, the Wiener process ${W}=\left\lbrace {W}(t), t \in [0,T] \right\rbrace $ and the fractional Brownian motion (fBm) with Hurst index $H\in(0,1)$, ${B}^H= \left\lbrace {B}^H(t), t \in [0,T] \right\rbrace,$ that is the Gaussian process with zero mean  and the covariance function $$\E {B}^H(t) {B}^H(s) =\frac12 \left(t^{2H}+s^{2H}- \lvert t-s \rvert ^{2H} \right), t,s\in [0,T].$$
Consider a mixed Gaussian process composed of $W$ and ${B}^H$  involving a non-random drift. More precisely, we consider the the mixed fractional Brownian motion with the drift, i.e., the process of the  form
\begin{equation}\label{z_all}
Z(t)=B^H(t)+W(t)+\int_{0}^{t}f(s)ds,\quad t \in [0,T],
\end{equation}
where  $f\in \Lambda_H$~is a non-random function and space $\Lambda_H$  will be specified below.
Consider the following problem: to annihilate the drift by the change of the probability measure. More precisely, to choose the other probability measure $\widetilde{\mathbb{Q}}$ such that
\begin{equation}\nonumber
 {Z}(t)=\widetilde{B}^H(t)+\widetilde{W}(t), \quad  t \in [0, T],
\end{equation}
where the Wiener process $\widetilde{W}$ and the fBm $\widetilde{B}^H$ are two  independent processes  under the measure $ \widetilde{\mathbb{Q}}$.
The main idea of the solution is to apply Girsanov theorem to fractional Brownian motion and Wiener process with drifts. To do so we need to distribute the trend $\int_{0}^{t} f(s)ds$  among $\widetilde{W}$ and $\widetilde{B}^H$ in some optimal way as follows
\begin{equation}\nonumber
\begin{gathered}
\widetilde{W}(t) = {W}(t) + \int_{0}^{t} f_1(s)ds, \quad \widetilde{B}^H(t) = {B}^H(t) + \int_{0}^{t} f_2(s)ds, t\in [0,T]\\
f_1(t)+f_2(t)= f(t) ,  t\in [0,T].
\end{gathered}
\end{equation}

In order to write down the  Radon-Nikodym derivative $\frac{d \widetilde{\mathbb{Q}}}{d\pr},$  let us recall here the (weighted) Riemann-Liouville fractional integrals, see papers \cite{jost} and  \cite{Melnikov,Makmish}.
\begin{definition}\label{def_Riemann}
The Riemann-Liouville left- and right-sided fractional integral of order $\mu > 0$ on $[0, T]$ is defined as
$$ \left(I_{0+}^{\mu} \varphi\right)(t) = \frac{1}{\Gamma(\mu)}\int_{0}^{t}(t-z)^{\mu-1}\varphi(z)dz, \quad t \in [0, T],$$
$$ \left(I_{T-}^{\mu} \varphi\right)(t) = \frac{1}{\Gamma(\mu)}\int_{t}^{T}(z-t)^{\mu-1}\varphi(z)dz, \quad t \in [0, T].$$
Define the weighted Riemann\,--\,Liouville integrals
\begin{equation}\label{oper_k1}
\begin{gathered}
(K_{0}^{H,*}\varphi)(t)=C_{1}(H) t^{H-\frac12}\left(  I_{0+}^{\frac12 -H} u^{\frac12 -H} \varphi(u)\right) (t), \\
(K_{T}^{H,*}\varphi)(t)=C_{1}(H) t^{\frac12-H}\left(  I_{T-}^{\frac12 -H} u^{H-\frac12 } \varphi(u)\right) (t),
\end{gathered}
\end{equation}
for  $H\in(0,1/2),$
and
\begin{equation}\label{oper_k2}
\begin{gathered}
(K_0^{H}\varphi)(t)=\frac{C_{1}^{-1}(H)}{\Gamma(H-1/2)} t^{H-\frac12}\int_0^{t}(t-u)^{H-\frac{3}{2}}u^{\frac12-H}\varphi(u)du,\\
(K_0^{H,*} \varphi) (t)=\frac{C_{1}(H) t^{H-\frac12}}{\Gamma(3/2-H)} \frac{d}{dt}\int_0^{t}(t-u)^{\frac{1}{2}-H}u^{\frac12-H}\varphi(u)du,\\
(K_T^{H}\varphi)(t)=\frac{C_{1}^{-1}(H)}{\Gamma(H-1/2)} t^{\frac12-H}\int_t^{T}(u-t)^{H-\frac{3}{2}}u^{H-\frac12}\varphi(u)du,\\
(K_T^{H,*}\varphi)(t)=-\frac{C_{1}(H)t^{\frac12-H}}{\Gamma(3/2-H)} \frac{d}{dt}\int_t^{T}(u-t)^{\frac{1}{2}-H}u^{H-\frac12}\varphi(u)du,
\end{gathered}
\end{equation}
in case $H\in(1/2,1).$ The constant $C_1(H)$ is equal to $\left(\frac{\Gamma{(2-2H)}}{2H\Gamma{(H+1/2)}\Gamma{(3/2-H)}}\right)^{\frac12}$.
\end{definition}
Since $B^H$ and $W$ are independent, we can write $\frac{d \widetilde{\mathbb{Q}}}{d\pr}=\frac{d{Q_W}}{d\pr}\frac{d{Q_{B^H}}}{d\pr},$ where
\begin{equation}\label{QW}
\frac{d{Q_W}}{d\pr}=  \exp \left\lbrace -\int_{0}^{T} f_1(t)d{W}(t)-\frac12  \|f_1\|^2_{L_2([0,T])}  \right\rbrace,\\
\end{equation}
according to standard Girsanov theorem,
and
\begin{equation}\label{QB}
\frac{d{Q_{B^H}}}{d\pr}= \exp \left\lbrace -\int_{0}^{T} (K_{0}^{H,*} f_{2})(t)d{B}(t)-\frac12  \| K_{0}^{H,*} f_2\|^2_{L_2([0,T])} \right\rbrace,
\end{equation}
according to Girsanov theorem for a fractional Brownian motion, see e.g. \cite[Lemma 3.1]{MishMakog}.
In the above representation $B=\lbrace{ B(t), t \geq0\rbrace}$ is a Brownian motion, related to $B^H$ as follows:
\begin{equation}\nonumber
 B(t)=\frac{\Gamma{(\frac32-H)}}{\Gamma{(H+\frac12)}}\int_{0}^{t} ( K_{T}^{H,*}  \mathbbm{1}_{[0, t]})(s)dB^H(s).
\end{equation}

The optimal drift distribution problem arose when solving two problems of different types, but all of them ultimately came down to solving a certain Fredholm integral equation of the second kind. Namely,  the paper \cite{Melnikov} was devoted to the problem of optimization of small deviation for mixed fractional Brownian motion with trend for the case $H\in(0,1/2),$ whereas
in the paper   \cite{MishMakog}, the same problem was  considered for the case $H\in(1/2,1).$ The paper \cite{Makmish} studied the problem of minimization of the entropy functional appearing under the distribution of the drift,  and this problem was studied for $H\in(0,1/2)$. Now our goal is twofold: first, to present the existing results from \cite{Melnikov, MishMakog, Makmish} and second, to demonstrate how to reduce the problem of minimization of entropy functional in the case $H\in(1/2,1).$

\subsection{How to reduce the problem of the optimization to the integral equation}\label{subsec2}

Let us start with the small deviations of a mixed fractional Brownian motion with trend. We are interested in  the asymptotic  $\pr\{|B^H(t)+W(t)+\int_0^t f(s)ds|\leq \varepsilon g(t),t\in [0,T]\},$ as $\varepsilon\to 0.$
After passing to the measure $\widetilde{\mathbb{Q}},$ we have from \cite[Lemma 3.3]{Makmish} and  \cite[Lemma 3.3]{Melnikov} the lower bound for this probability
\begin{align*}
&\pr\left\{\left|B^H(t)+W(t)+\int_0^t f(s)ds\right|\leq \varepsilon g(t),t\in [0,T]\right\}\\
&\geq \exp\left(-\frac12\|f_1\|^2_{L_2([0,T])}-\frac12\| K_{0}^{H,*}f_2\|^2_{L_2([0,T])} \right\} \widetilde{\mathbb{Q}}\left\{\left|\widetilde{B}^H(t)+\widetilde{W}(t)\right|\leq \varepsilon g(t),t\in [0,T]\right\}.
\end{align*}
Therefore, the maximization of its right-hand side leads to the following optimization problem
\begin{equation}
\label{minpr}
\|f-x\|^2_{L_2([0,T])}+ \| K_{0}^{H,*} x\|^2_{L_2([0,T])} \xrightarrow{x\in \Lambda_H} \min
\end{equation}
where $\Lambda_H=L_2([0,T]), $ if $H\in (0,1/2].$ If $H\in (1/2,1),$ then $\Lambda_H$ consists of all functions $h:[0,T]\to\R$ for which  there exist $\phi_1,\phi_2\in L_2([0,T])$  such that $u^{\frac12-H}h(u)=I^{H-\frac12}_0 \phi_1(u),$ and $u^{H-\frac12}K^{H,*}_0h(u)= I^{H-\frac12}_0 \phi_2(u).$ Furthermore, if $x_0$ is a minimizator in \eqref{minpr}, then
\begin{align*}
&\pr\left\{\left|B^H(t)+W(t)+\int_0^t f(s)ds\right|\leq \varepsilon g(t),t\in [0,T]\right\} \stackrel{\varepsilon \to 0}{\sim}\\
& \exp\left(-\frac12\|f-x_0\|^2_{L_2([0,T])}-\frac12\| K_{0}^{H,*}x_0\|^2_{L_2([0,T])} \right\} \widetilde{\mathbb{Q}}\left\{\left|Z(t)\right|\leq \varepsilon g(t),t\in [0,T]\right\}.
\end{align*}

Now consider the minimization of entropy functional, which can be formulated as follows:  define the functions $f_1$ and  $f_2$ in \eqref{QW} and \eqref{QB}, which minimize the entropy-type  functional
\begin{equation}\label{h_q_p}
\mathsf{H}_{1}(\pr, Q_{W}, Q_{B^H} )=\E \left[ \left(\frac{dQ_{W}}{d\pr} \log \frac{dQ_{W}}{d\pr} \right) \right]+\E \left[ \left(\frac{dQ_{B^H}}{d\pr} \log \frac{dQ_{B^H}}{d\pr} \right) \right].
\end{equation}
The next result was proved in \cite{Makmish} for the case $H\in (0,1/2),$ but the proof remains the same for all $H\in (0,1).$
\begin{lemma}  Entropy functional $\mathsf{H}_{1}(\pr, Q_{W}, Q_{B^H} )$  could be represented as
\begin{equation}\label{h_alpha}
\mathsf{H}_{1}(\pr, Q_{W}, Q_{B^H} )= \frac12\|f_1\|^2_{L_2([0,T])}+ \frac12\| (K_{0}^{H,*}f_2)\|^2_{L_2([0,T])}.
\end{equation}
\end{lemma}
Thus, the minimization of \eqref{h_q_p} is equivalent to the optimization problem \eqref{minpr}.

It was shown in \cite{Melnikov} for $H\in (0,1/2)$ and in \cite{Makmish} for $H\in (1/2,1)$, that the minimization in \eqref{minpr} is a solution of the following fractional integral/differential equation
\begin{equation}\label{h_new1}
x(t) + \left[K_{T}^{H,*} K_{0}^{H,*}x\right](t)=f(t) , \quad t \in [0,T].
\end{equation}
The existence and uniqueness of solution for equation \eqref{h_new1} was proved  in  Theorem 3.9 \cite{Melnikov} in the case of Hurst index $H \in (0, \frac12).$ When $H\in( \frac12,1),$ it was proved in \cite{MishMakog}, that \eqref{h_new1}  is equivalent to
\begin{equation}\label{h_new2}
[f-x](t) + \left[K_{0}^{H} K_{T}^{H}(f-x)\right](t)=f(t) , \quad t \in [0,T]
\end{equation}
for $f\in \Lambda_H,$  which has a unique solution $f-x_0\in \Lambda_H.$
After applying definition of weighted Riemann\,--\,Liouville integral \eqref{def_Riemann}, the operators $K_{T}^{H,*} K_{0}^{H,*},$ $H\in (0,1/2),$ and $K_{0}^{H} K_{T}^{H},$ $H\in (1/2,1)$ have the form of integral operator with the kernel $C_2(H)\varkappa_H,$ where
\begin{equation}\label{kern_1}
\varkappa_H(t,v)=\begin{cases}
(tv)^{1/2-H}\int_{t\vee v}^T (z-t)^{-1/2-H} z^{2H-1} (z-v)^{-1/2-H}dz ,& H \in (0,1/2),\\
(tv)^{H-1/2} \int_0^{t \wedge v} (t-z)^{H-3/2}z^{1-2H}(v-z)^{H-3/2}dz, &  H \in (1/2,1).
\end{cases}
\end{equation}
and $C_2(H)= \left( \frac{C_1(H)}{\Gamma(|H-1/2|)} \right) ^2.$ Thus, the both optimization problems reduces the solution of the  Fredholm integral equation of the second kind, which can be represented as
\begin{equation}\label{Feq}
	x(t) + C_2(H) \int_0^T {\varkappa_H}(t,v) x(v) dv= f(t),
\end{equation}

In this paper, we study further equation \eqref{Feq} and prove in the next lemma that kernel \eqref{kern_1} can be significantly simplified.

Let  $B(\alpha,\beta)$  be the Beta function and $B(x,\alpha,\beta)$ be an  incomplete beta function defined for $x\in [0,1],$ given by
$B(x,\alpha,\beta)=\int_0^x y^{\alpha-1}(1-y)^{\beta-1}dy,$ $\alpha, \beta >0$.
\begin{lemma}
 The kernel \eqref{kern_1} equals
\begin{enumerate}[(i)]
\item for $H \in (0,1/2)$
\begin{equation}\nonumber
\varkappa_H(t,v)=\frac{1}{|t-v|^{2H}}B\left(\frac{T/(t\vee v)-1}{T/(t\wedge v)-1},\frac{1}{2}-H,2H\right), \quad t,v\in [0,T]
\end{equation}
where numerator is bounded on $[0,T]^2$, meanwhile, has no limit at points $(0,0)$ and $(T,T)$.
\item for $H \in (1/2,1)$
\begin{equation}\nonumber
\varkappa_H(t,v)=\frac{1}{|t-v|^{2-2H}}B\left(H-\frac12,2-2H\right),  \quad t,v\in [0,T].
\end{equation}
\end{enumerate}
\end{lemma}
\begin{proof}
Item (i): it is easy to see that kernel $\varkappa_H$ is symmetric, so that it is enough to consider only the case $v<t.$ Then
\begin{equation}\label{twoint1}
\begin{gathered}
\varkappa_H(t,v)=(tv)^{\frac{1}{2}-H}\int_t^T (z-t)^{-\frac{1}{2}-H} z^{2H-1} (z-v)^{-\frac{1}{2}-H}dz\\
=(tv)^{\frac{1}{2}-H}\int_t^T \left(1-\frac{t}{z}\right)^{-\frac{1}{2}-H}\left(1-\frac{v}{z}\right)^{-\frac{1}{2}-H} \frac{dz}{z^2}.
\end{gathered}
\end{equation}
In turn, transform the last integral in \eqref{twoint1} with the change of variables $\frac{z-v}{z}=\frac{t-v}{t(1-x)}.$ Then
\begin{equation}\nonumber
\frac{1}{z}=\frac{v-tx}{vt (1-x)},\quad
\frac{dz}{z^2}=\frac{(t-v)dx}{vt (1-x)^2},\quad
1-\frac{t}{z}=\frac{(t-v)x}{v(1-x)}
\end{equation}
and we get
\begin{align*}
\varkappa_H(t,v)&=(tv)^{\frac{1}{2}-H}\int_t^T \left(1-\frac{t}{z}\right)^{-\frac{1}{2}-H}\left(1-\frac{v}{z}\right)^{-\frac{1}{2}-H} \frac{dz}{z^2}\\
&=(tv)^{\frac{1}{2}-H}\int_0^{\frac{v(T-t)}{t(T-v)}} \left(\frac{(t-v)x}{v(1-x)}\right)^{-\frac{1}{2}-H}\left(\frac{t-v}{t(1-x)}\right)^{-\frac{1}{2}-H} \frac{t-v}{vt (1-x)^2}dx\\
&=(t-v)^{-2H}\int_0^{\frac{v(T-t)}{t(T-v)}} x^{-\frac{1}{2}-H}(1-x)^{2H-1} dx.
\end{align*}
Item (ii): it was proved  in \cite{MishMakog} that kernel $\varkappa_H$ is symmetric and non-negative, consequently, we  consider only the case $t<v.$ Then
\begin{equation}\label{twoint2}
\begin{gathered}
\varkappa_H(t,v)=(tv)^{H-\frac12}\int_0^{t}(t-z)^{H-\frac{3}{2}}z^{1-2H} (v-z)^{H-\frac{3}{2}}dz\\
=(tv)^{H-\frac12}\int_0^t \left(\frac{t}{z}-1\right)^{H-\frac32}\left(\frac{v}{z}-1\right)^{H-\frac32} \frac{dz}{z^2}.
\end{gathered}
\end{equation}
Introduce the similar change of variables for the second integral in \eqref{twoint2} as $\frac{v-z}{z}=\frac{v-t}{t(1-x)}$. Then
\begin{equation}\nonumber
\frac{1}{z}=\frac{v-tx}{vt (1-x)}, \quad
\frac{dz}{z^2}=\frac{(t-v)dx}{vt (1-x)^2},\quad
\frac{t}{z}-1=\frac{(v-t)x}{v(1-x)},
\end{equation}
and we obtain
\begin{align*}
\varkappa_H(t,v)&=(tv)^{H-\frac12}\int_0^t \left(\frac{t}{z}-1\right)^{H-\frac32}\left(\frac{v}{z}-1\right)^{H-\frac32} \frac{dz}{z^2}\\
&=(tv)^{H-\frac12}\int_1^0 \left(\frac{(v-t)x}{v(1-x)}\right)^{H-\frac32}\left(\frac{v-t}{t(1-x)}\right)^{H-\frac32} \frac{t-v}{vt (1-x)^2}dx\\
&=(t-v)^{2H-2}\int_0^1 x^{H-\frac32}(1-x)^{1-2H} dx.
\end{align*}
The lemma is proved.
\end{proof}
\section{Approximation theorem for integral operator}\label{sec3}
We start with very simple auxiliary approximation result for the sequence of operators.  Let $\mathcal{H}$ be a real Hilbert space.
\begin{lemma}\label{lem:oper} Let $A:\mathcal{H}\rightarrow\mathcal{H}$ be a compact   positive operator, and let $A_n, n\geq 1:\mathcal{H}\rightarrow\mathcal{H}$ be a sequence of compact operators with spectrum $\sigma(A_n)$ such that $\|A-A_n\|\rightarrow 0$ as $n\rightarrow \infty$.

Then the spectrum $\sigma(A_n)$ is asymptotically included into $\mathbb{R}^+$, in the sense that  $$\dist(\sigma(A_n), \mathbb{R}^+)\rightarrow 0\;\text{as}\;  n\rightarrow \infty.$$

\end{lemma}
\begin{proof} Suppose the contrary: let there exist some $\beta>0$ and subsequent $\{n_k,k\geq 1\}$ such that $n_k\to\infty$ as $k\to\infty$ and $\dist(\sigma(A_{n_k}), \mathbb{R}^+)\ge \beta>0,$ $k\ge 1.$ Then there exists  $\lambda_{n_k}\in\sigma(A_{n_k})$ such that $\dist(\lambda_{n_k}, \mathbb{R}^+)\ge \frac{\beta}{2},\;  k\ge 1.$ However,  let $\lambda_{n_k}\in\sigma(A_{n_k})$ be any sequence of eigenvalues of operators $A_{n_k}$, $k\ge 1$, and let
$$A_{n_k}x_{n_k}=\lambda_{n_k}x_{n_k},\, \|x_{n_k}\|=1,\, k\ge 1.$$
Let us write the following obvious equalities:  $$Ax_{n_k}=\lambda_{n_k}x_{n_k}+(A-A_{n_k})x_{n_k},$$
whence
$$\langle Ax_{n_k}, x_{n_k} \rangle=\lambda_{n_k}+\langle(A-A_{n_k})x_{n_k},x_{n_k}\rangle. $$
Furthermore, $|\langle(A-A_{n_k})x_{n_k},x_{n_k}\rangle|\leq \|A-A_{n_k}\|\rightarrow 0$ as $k\rightarrow \infty$, and $\langle Ax_{n_k}, x_{n_k}\rangle\ge 0$, and it immediately follows that $\dist(\lambda_{n_k}, \mathbb{R}^+)\rightarrow 0\;\text{as}\;  k\rightarrow \infty.$ The resulting contradiction proves the theorem. \end{proof}
Now we are apply Lemma \ref{lem:oper}  to the sequence of the integral operators in the space $L_2([0,T])$ for some $T>0$.
Namely, consider the integral operator $A$ defined by its kernel  function $K(t,s)$ via formula
\begin{equation}\label{integral:operator}
(Ax)(t)= \int_0^T K(t,s) x(s)ds, 0 \leq t \leq T,
\end{equation}
where $x$ is taken from some space of functions defined on $[0,T]$, and this space will be specified later. We assume that $K$ has a singularity, more precisely, it has the form
\begin{equation}\label{Fredh_kern}
K(t,s)=\frac{L(t,s)}{|t-s|^{\nu}}
\end{equation}
where $L$ is a bounded function on $[0,T]^2$ and $0 < \nu <1$. Let us recall  the following general statement from \cite[p.397, item 6.4]{kant}.
\begin{proposition}\label{prop:1} Let kernel $K$ of operator $A$ from \eqref{integral:operator} satisfy the following conditions: there exist $C_1,C_2>0$, $r_1>1,r_2>1$ and $p>1$  such that $p-r_1(p-1)<r_2<p$ and for which
  $$\int_0^T|K(t,s)|^{r_1}ds<C_1 \text{ and } \int_0^T|K(t,s)|^{r_2 }dt<C_2.$$
  Then operator $A$ is a compact operator from $L_p([0,T])$ into $L_p([0,T])$ with the norm $$\|A\|\leq C_1^{\frac{p-r_2}{pr_1}}C_2^{\frac{1}{p}}.$$
  \end{proposition}
  \begin{corollary}
  \label{cor:1}Consider the integral operator \eqref{integral:operator} with the kernel \eqref{Fredh_kern}, where $L$ is a bounded function on $[0,T]^2$ and $0 < \nu <1$. Then we can put $r_1=r_2=\frac{1}{\nu}-\delta>1$ for any $0<\delta<\frac{1}{\nu}-1$, and additionally we can choose $\delta$ in such a way that $\frac{1}{\nu}-\delta<2$. Then we can put $p=2$ and all conditions of Proposition \ref{prop:1} will be fulfilled. Therefore $A$ is a compact operator from $L_2([0,T])$ into $L_2([0,T])$, and so  $L_2([0,T])$ is a space that was claimed to be specified later.\end{corollary}

Now, consider   the Fredholm integral equation of the second kind
\begin{equation}\label{Fredh_gen}
x(t)+ (Ax)(t)=g(t), \, t\in [0,T],
\end{equation}
where $g(t) \in L_2 ([0, T])$ is a given function, and the integral operator $A$  has the form \eqref{integral:operator} and the kernel is taken from \eqref{Fredh_kern}. The standard situation is when the function $L$ is continuous. However, in the applications
which we will consider later, function $L$ will be bounded however, it may have a finite number of fatal discontinuities, i.e., points in which the limit of function $L$ does not exist. We call such kernels as the kernels with additional singularity. In this connection, let us prove an auxiliary result concerning the possibility of approximation of the kernel $L$ by the respective kernels with continuous numerators.
\begin{lemma}\label{approx} Let the function $L=\{L(t,s),t,s\in[0,T]\}$ be  bounded, $|L(t,s)|\le L$, and continuous a.s. except finite number of points. Let $K(t,s)=L(t,s)|t-s|^{-\nu},$ $K_n(t,s)=L_n(t,s)|t-s|^{-\nu}, t,s\in [0,T]$ for $\nu\in (0,1).$ Let $0<\delta<1/\nu$ be fixed. Then there exists a sequence  of totally bounded continuous functions $L_n=\{L_n(t,s),t,s\in[0,T]\}$ (we can take the same constant $L$ for them), such that
$$\int_0^T|K(t,s)-K_n(t,s)|^{\frac{1}{\nu}-\delta}dt<C_{1,n},\; \text{ and }\; \int_0^T|K(t,s)-K_n(t,s)|^{\frac{1}{\nu}-\delta}ds<C_{2,n},
$$
where $C_{1,n},C_{2,n}$ depend  on $\delta, L$ and $C_{1,n},C_{2,n}\rightarrow 0$ as $n\rightarrow \infty$.
\end{lemma}
\begin{proof} Let us describe  the construction of continuous function $L_{n}$.
Let $x_0=(t_0,s_0)$ be one of the  points of discontinuity,   let   $m$ be the total number of such points, and let $n$ is sufficiently large.
Each point of discontinuity, in particular, $x_0$, can be surrounded   by sufficiently small closed box $B(x_0,r_n)=\{y\in[0,T]:\|y-x_0\|_{\max}\leq r_n\},$ where  $r_n=1/n$ and $\|\cdot\|_{\max}$ is the maximum norm in Euclidean space. Assume that $L_n=L$ outside the union of these small boxes, so, it is necessary to determine $L_n$ only inside each box. Let us put for $x=(t,s) $ and
\begin{equation}\label{Ln:def}
 L_n(x)=\frac{\|x-x_0\|_{\max}}{r_n}      L\left( \frac{r_n (x-x_0)}{\|x-x_0 \|_{\max}}+x_0 \right),\;  x\in B(x_0,r_n).
 \end{equation}
The range of the values of $L_n$ does not exceed the range of values of $L$.  The point $\frac{r_n (x-x_0)}{\|x-x_0 \|_2}+x_0$ is situated on the square $$S=\{y\in [0,T]^2: \|y-x_0\|_{\max}=r_n\},$$
 therefore every $L_n$ is a continuous function, $L_n\in C([0,T]^2).$ Moreover,
 $$\sup_{x\in B(x_0,r_n)}|L_n(x)|\leq \sup_{x\in S(x_0,r_n)}|L(x)|,$$ and
 $$\sup_{x\in [0,T]^2}|L_n(x)|\leq \sup_{x\in[0,T]^2}|L(x)|.$$  Thus,  $L_n$ are totally bounded.

 Now, denote $D_n\subset [0,T]$ the projection of the union of the small boxes surrounding the points of discontinuity of $L$, on $[0,T]$. Evidently, the total Lebesgue measure of $D_n$ does not exceed $\frac{m}{n}\rightarrow 0$ as $n\rightarrow\infty$. Therefore,
 \begin{equation*}\begin{gathered}C_{1,n}:=\sup_{s\in[0,T]}\int_0^T|K(t,s)-K_n(t,s)|^{\frac{1}{\nu}-\delta}dt\\ \le \sup_{(t,s)\in [0,T]^2}|L(t,s)| \sup_{s\in[0,T]}\int_{D_n}|t-s|^{-1+\nu\delta}dt\rightarrow 0,\end{gathered}\end{equation*}
  as $n\rightarrow\infty$, and $C_{2,n}$ can be introduced and treated similarly, whence the proof follows.
 \end{proof}
 \begin{remark} Of course, there can be different ways of construction of the functions $L_n$.  In what follows, for us will be important to construct them in such a way that the approximating operators be self-adjoint.
 \end{remark}
 \begin{remark}
 We can apply Proposition \ref{prop:1} and Corollary \ref{cor:1} to operator $(A_n-A)x:=\int_0^T [K_n(\cdot,s)-K(\cdot,s)]x(s)ds, x\in L_2([0,T])$ for $\left(\frac{1}{\nu}-2\right)_+<\delta<\frac{1}{\nu}-1.$  Then $A_n-A$ is a compact operator from $L_2([0,T])$ to $L_2([0,T])$ with  $$\|A_n-A\|\leq C_{1,n}^{\frac{\nu}{1-\delta\nu}-\frac{1}{2}}C_{2,n}^{\frac{1}{2}}\to 0,\quad\text{ as  }n\to \infty.$$
 \end{remark}
 Now, let us return to equation \eqref {Fredh_gen} and specify the assumption regarding  the integral operator $A$.
\begin{lemma}
Let the integral operator $A$ is compact from $L_2([0,T])$ into $L_2([0,T])$ and positive, in particular, self-adjoint, and let  $g \in L_2([0,T])$. Then there exists a unique function $x=x(t), t \in [0,T], x\in L_2([0,T])$ which satisfies \eqref {Fredh_gen}.
\end{lemma}
\begin{proof} It is just sufficient to mention that all eigenvalues of operator $A$ are real and nonnegative, therefore the respective homogeneous equation has only trivial solution, and the proof immediately follows from the Fredholm alternative.\end{proof}
Now, let us establish the main result of this section, namely, the theorem on the approximation  of solution of integral equation with the kernel containing additional singularity by the solutions of the integral equations whose kernels are of type \eqref{Fredh_kern}, but the numerator is continuous.
\begin{theorem}\label{ContKern}
Let $K(t,s)$  be the kernel defined in \eqref{Fredh_kern}, where the numerator $L(t,s)$ has the following properties
\begin{enumerate}[(i)]
\item $L$ is bounded and symmetric.
\item $L$ is continuous, except finite number of points.
\item $L$ is a  positively definite kernel.
\end{enumerate}
Let $x\in L_2([0,T])$ be a unique solution of equation \eqref{Fredh_gen}.
Then the sequence of functions  $L_{n} (t,s)$ satisfying conditions of Lemma \ref{approx} can be chosen in such a way that the respective integral operators are self-adjoint, for sufficiently large  $n\ge 1$   the equation
$$ x_n(t)+ \left(A_nx\right)(t)=g(t),\; (A_nx)(t)= \int_0^T K_n(t,s) x(s)ds, 0 \leq t \leq T,\; K_n(t,s)=\frac{L_n(t,s)}{|t-s|^{\nu}},$$
has  a unique solution $x_n\in  L_2([0,T])$,
and \begin{equation}\nonumber
\| x_{n} -x \|_{L_2([0,T])}\rightarrow 0,\;\text{as}\;n\rightarrow \infty.
\end{equation}
\end{theorem}

\begin{proof} Let us choose $L_n$ according to Lemma \ref{approx}. Then for the respective integral operators, according to
Lemma \ref{approx} and Proposition \ref{prop:1} we have that $\|A-A_n\|\rightarrow 0$ as $n\rightarrow\infty$.
Show that $L_n$ is symmetric which yields that $A_n$ is self-adjoint. Let $D$ be the union of  the points of discontinuity of $L.$ Denoting by $x^*=(x_2,x_1)$ for $x=(x_1,x_2)\in [0,T]^2$ we have from symmetry of $L$ that if $x_0\in D,$ then $x^*_0\in D$ as well. Then
from \eqref{Ln:def} we have  $L_n(x^*)=L(x)=L(x^*)=L_n(x^*)$ for all points $x$ outside  $\cup_{x_0\in D}B(x_0,r_n).$ Assume that $n$ is large enough that all $B(x_0,r_n), x_0\in D$ are disjoint.  Let $x\in B(x_0,r_n)$ for some $x_0\in D,$ then $x^*\in B(x^*_0,r_n)$ and it follows from the construction of $L_n$ that $L_n(x^*)=\frac{\|x^*-x^*_0\|_{\max}}{r_n}      L\left( \frac{r_n (x^*-x^*_0)}{\|x^*-x^*_0 \|_{\max}}+x^*_0 \right)=\frac{\|x-x_0\|_{\max}}{r_n}      L\left( \left(\frac{r_n (x-x_0)}{\|x-x_0 \|_{\max}}+x_0\right)^* \right)=L_n(x).$

According to Lemma \ref{lem:oper}, spectrum $\sigma(A_n)$ is asymptotically included into $\mathbb{R}^+$. Moreover, these operators are compact from $L_2([0,T])$ into $L_2([0,T])$. It means that for sufficiently large  $n\ge 1$ integral equation $ x_n(t)+ \left(A_nx\right)(t)=g(t)$ has the unique solution $x_n\in L_2([0,T])$. Now, consider the difference
\begin{equation}\nonumber
\begin{gathered}
x(t)- x_{n} (t)= - \int_0^T K(t,s) x(s)ds+ \int_0^T K_{n} (t,s) x_{n} (s)ds\\
= -\int_0^T \left[ K(t,s) -K_{n} (t,s) \right] x(s)ds - \int_0^T K_{n} (t,s) \left[ x(s)- x_{n}(s)\right]ds.
\end{gathered}
\end{equation}
Denoting $x(t)-x_{n}(t)= y_n(t)$, we get the    Fredholm integral equation of the form

\begin{equation}\nonumber
y_{n} (t)+\int_0^T K _n (t,s)  y_{n} (s) ds=P_{n}(t),
\end{equation}
where $P_{n}(t)= \int_0^T \left[ K_{n}(t,s) -K (t,s) \right] x(s)ds.$ Since $\|A-A_n\|\rightarrow 0$ as $n\rightarrow\infty$, we have that $\| P_{n}\|_{L_2([0,T])}\rightarrow 0$ as $n\rightarrow\infty$. Now, denote the resolvent operator at point $\lambda$: $R_{A_n}(\lambda)=(A_n - \lambda I) ^{-1}$.
We shall use the following fact (see e.g. \cite[Theorem 5.8]{Hislop}): let $B$ be a self-adjoint operator, and let $\lambda \in \rho (B)$, where $\rho (B)=\mathbb{R}\setminus\sigma(B)$. Then
\begin{equation}
\| R_B (\lambda)\| \leq \frac{1}{\dist(\lambda, \sigma(A))}.
\end{equation}
In our case $\lambda=-1$, and we know that the spectrum $\sigma(A_n)$ is asymptotically included into $\mathbb{R}^+$. Therefore for sufficiently large $n$ $\|R_{A_n}(-1)\|\leq C$, and we get that
$$\|y_n\|_{L_2([0,T])}\leq C\| P_{n}\|_{L_2([0,T])}\rightarrow 0$$ as $n\rightarrow\infty$. The proof is concluded.
 \end{proof}

\section{Numerical solution}\label{sec4}
\subsection{Description of the numerical method}\label{subsec11}

In our paper, we use a modified product-integration method, as proposed in \cite{Prem} for weakly singular kernels. Under this method we assume that the kernel $K$ is factorized as  kern  have
\begin{equation}\nonumber
K(t,s)=u(t,s)L(t,s),
\end{equation}
where $u(t,s)$ is singular in $\{(t,t),t\in (0,T)\}$ and $L(t,s)$ is a regular function of its arguments.  Then the Fredholm integral equation of the second kind is written as
\begin{equation}\label{8_4_4}
x(t)=f(t)+\int_0^T u(t,s)L(t,s)x(s)ds.
\end{equation}
The main mechanism of product-integration method is presented in Section 4.2 of \cite{Prem}.

There are several different modifications of this method, see e.g. the corresponding chapters in books \cite{Atk1976, Baker78, Hoog73, Delves85}. In particular, the authors of \cite{Hoog73} prove that their modification works for a kernel of the form
\begin{equation}\nonumber
K(t,s) = u (t,s)  L (t,s), t,s\in [0,T]
\end{equation}
where the functions $u (t,s)$ and $L (t,s)$ satisfy the following conditions
\begin{enumerate}[(i)]
\item $L (t,s) \in C[0,T],$
\item $\int_0^T |u(t,s)|ds < \infty,$
\item $\lim_{|t_1-t_2| \to 0} \int_0^T |u(t_1,s) - u(t_2,s) | ds =0$  uniformly in $t_1$ and $t_2$.
\end{enumerate}
The above conditions hold true for the kernels of potential-type with $u(t,s)=|t-s|^{-\nu},\nu\in (0,1).$  Consequently, general product-integration method works in the case of a continuous numerator $L$ (and hence all its modifications).

In this article, by computational reasons, we have chosen the modification proposed by B. Neta in \cite{Neta}.  Here we present its main steps.
We start with a given number $N \geq 1$ of equally spaced points
$0= t_1 < t_2 < ... < t_N =T$ and $h_i= t_{i+1}-t_i=\frac{T}{N-1}$.  By the product-integration rule
\begin{equation}\nonumber
\int_{t_i}^{t_{i+1}}L(t,s)|t-s|^{-\nu}x(s)ds\approx \int_{t_i}^{t_{i+1}} \frac{L_i(t)(t_{i+1}-s)x_i + (s-t_i)L_{i+1}(t)x_{i+1}}{ t_{i+1}-t_i}|t-s|^{-\nu}ds,
\end{equation}
where $L_i(t)=L(t,t_i)$ and $x_i=x(t_i)$. Then we have
\begin{equation}\label{f_sum}
x(t)=\sum_{i=1}^{N-1} \left(L_i(t) x_i \psi_{i, i+1}^{1}(t)+ L_{i+1} x_{i+1} \psi_ {i, i+1}^{2} (t)\right)+g(t),
\end{equation}
where weights are assigned as follows
\begin{equation}\label{psi_int}
\begin{gathered}
\psi_{i+1, i}^{1} (t) = \frac{1}{t_{i+1}-t_i} \int_{t_i}^{t_{i+1}} (t_{i+1}-s) |t-s|^{-\nu} ds,\\
\psi_{i-1, i}^{2} (t) = \frac{1}{t_i-t_{i-1}} \int_{t_{i-1}}^{t_i} (s-t_{i-1}) |t-s|^{-\nu} ds.
\end{gathered}
\end{equation}
Substituting $t=t_j$ in \eqref{f_sum} and combining two sums we obtain the following system of linear equations for approximation of integral equation \eqref{8_4_4}
\begin{equation}\label{fj_sum}
x_j=\sum_{i=1}^N L_{j,i} \left( \psi_{i, i+1, j}^{1}+ \psi_{i-1, i, j}^{2} \right) x_i  +f_j, j = 1, \ldots, N,
\end{equation}
where $L_{j,i}=L(t_i,t_j),$ $\psi_{i, j, k} ^{l} = \psi_{i, j}^l (t_k)$ for $l =1,2,$ and $f_j=f(t_j)$.
In \eqref{fj_sum} we assume that $\psi_{N, N+1, j}^{1} = \psi_{0, 1, j}^{2} = 0$, for all $j$.
The integrals in \eqref{psi_int} are evaluated exactly and the values of $\psi_{i j k} ^{l}$ are computed separately for the cases $k=i$ and $k=j$. It can be shown that
\begin{align*}
&\psi_{i, i+1, i}^{1}= \frac{(t_{i+1}-t_i)^{1-\nu}}{(1-\nu)(2-\nu)}, \quad \psi_{i, i+1, i+1}^{1}= \frac{(t_{i+1}-t_i)^{1-\nu}}{2-\nu},\\
&\psi_{i, i+1, j}^{1}= \frac{|t_{i+1}-t_j|}{t_{i+1}-t_i} \frac{|t_{i+1}-t_j|^{1-\nu} - |t_i-t_j|^{1-\nu}}{1-\nu} - \frac{|t_{i+1}-t_j|^{2-\nu} - |t_i-t_j|^{2-\nu}}{(2-\nu)(t_{i+1}-t_i)}, j \neq i, i+1,\\
&\psi_{i-1, i, i-1}^{2}= \frac{(t_i-t_{i-1})^{1-\nu}}{2-\nu}, \quad \psi_{i-1, i, i}^{2}= \frac{(t_i-t_{i-1})^{1-\nu}}{(1-\nu)(2-\nu)},\\
&\psi_{i-1, i, j}^{2}= \frac{|t_i-t_j|^{2-\nu}- |t_{i-1}-t_j|^{2-\nu}}{(t_i-t_{i-1})(2-\nu)}
-\frac{|t_j-t_{i-1}|}{t_i-t_{i-1}} \frac{|t_i-t_j|^{1-\nu} - |t_{i-1}-t_j|^{1-\nu}}{1-\nu}, j \neq i-1, i.
\end{align*}
The system \eqref{fj_sum} can be written in matrix form
\begin{equation}\label{linear_eq}
\mathbf{X}_N = \mathbf{K}_N \textbf{X}_N + \textbf{G}_N,
\end{equation}
where $\mathbf{X}_N$ and $\mathbf{G}_N$ are vectors whose components are $x_i$ and $g_i,$ $i=1,\ldots,N$ respectively, and $\mathbf{K}_N$ is $N\times N$-matrix with components $K_{i,j}= L_{i,j}\left(\psi_{j, j+1, i}^{1}+ \psi_{j-1, j, i}^{2} \right), i,j=1\ldots,N.$
The approximate solution  is obtained by solving a linear system of algebraic equations \eqref{linear_eq}.

Let us now describe the application of the presented method for solving of equation \eqref{Feq}
\begin{align}
 &x(t) + C_2(H) \int_0^T  \frac{L(t,v)}{|t-v|^{2H}} dv= f(t),  t\in [0,T], \text{ if } H\in \left(0,\frac{1}{2}\right)  \\
 \nonumber &\text{ with } L(t,v)=B\left(\frac{T/(t\vee v)-1}{T/(t\wedge v)-1},\frac{1}{2}-H,2H\right);\\
 &x(t) + C_2(H)B\left(H-\frac12,2-2H\right) \int_0^T  \frac{dv}{|t-v|^{2-2H}}  =f(t),  t\in [0,T], \text{ if } H\in \left(0,\frac{1}{2}\right).
\end{align}
Thus, we get immediately system of equations \eqref{fj_sum} with $L_{i,j}=1$ in the case $H\in \left(0,\frac{1}{2}\right).$
If $H<1/2,$ the numerator $L$ has additional singularity in two points $(0,0)$ and $(T,T)$ and we solve system of equations \eqref{fj_sum} with $L_{i,j}=L_n(t_i,t_j),$ where the approximation $L_n$ of $L$, given by \eqref{Ln:def}. Note that $L_n(t_i,t_j)=L(t_i,t_j),$ for $i,j\neq (1,1)$ and $i,j\neq (N,N)$ if $n\geq\frac{N-1}{T}.$ In this case, $L_n(t_1,t_1)=L_n(t_N,t_N)=0.$ For simplicity, we put always $n=\frac{N-1}{T}$ in our computations.

\subsection{Numerical illustrations}\label{subsec12}
In this section we provide several numerical experiments of solving equation \eqref{Feq} performed for different values $H$ and functions $f$.

First, we provide in Figure \eqref{Fig1} the graph of $L(t,v)=B\left(\frac{T/(t\vee v)-1}{T/(t\wedge v)-1},\frac{1}{2}-H,2H\right),$ $(t,v)\in[0,T]^2$ in the case $H<1/2$ for better understanding of singularity in the numerator.
\begin{figure}
    \centering
    \begin{minipage}{0.3\textwidth}
		\center{\includegraphics[width=\textwidth]{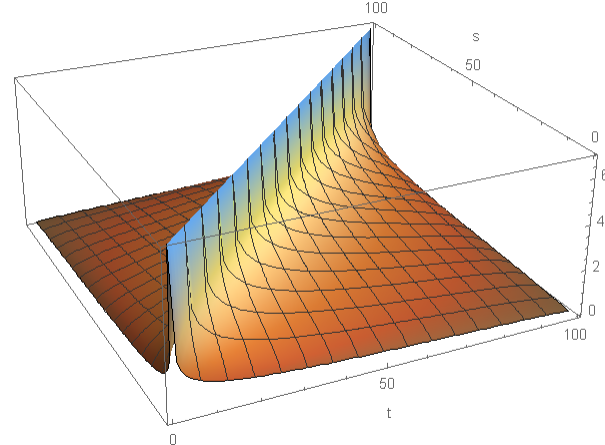}}
		\center{$H=0.1$}
	\end{minipage}
    \begin{minipage}{0.3\textwidth}
		\center{\includegraphics[width=\textwidth]{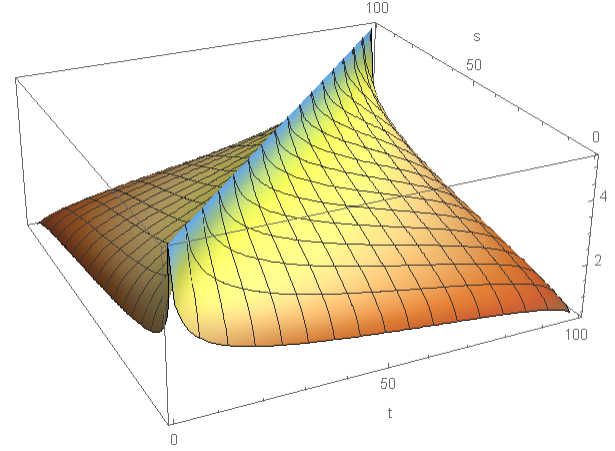}}
		\center{$H=0.25$}
	\end{minipage}
    \begin{minipage}{0.3\textwidth}
		\center{\includegraphics[width=\textwidth]{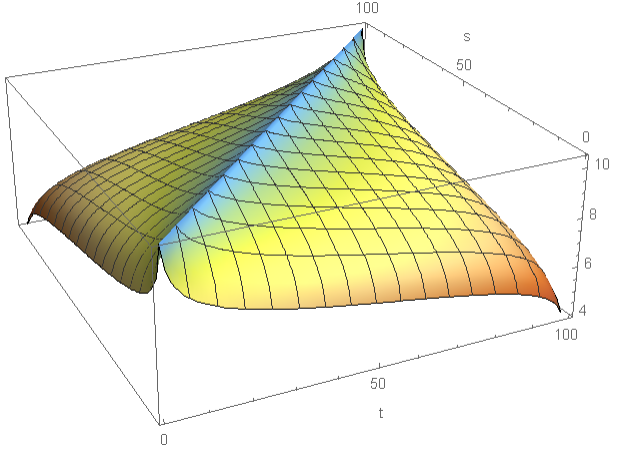}}
		\center{$H=0.4$}
	\end{minipage}
    \caption{The values of the numerator $L(t,v)$ for different values of $H.$}
    \label{Fig1}
\end{figure}

\begin{example}
Consider  equation \eqref{Feq} with the simple linear function $f(t)=t,t\in [0,1].$ We take $N=500,$ and present approximated solution $x$ with $H=0.05,0.2,0.4,0.49$ in Figure \eqref{Fig2}. For $H=0.51,0.6,0.8,0.95$ we present the graph of $f-x.$ Thus we obtain the graphs of solution $f_1$ of minimization problem \eqref{minpr}. Note that operator $A$ corresponding to equation \eqref{Feq} tends to identity operator, as was shown in \cite{Makmish}. Therefore, the solution $f_1$ tends to $\{t/2,t\in [0,1]\},$ which is confirmed by numerical solutions from Figure \ref{Fig2}.
\begin{figure}[htbp]
 \centering
 \label{Fig2}
  \includegraphics[width=0.47\linewidth]{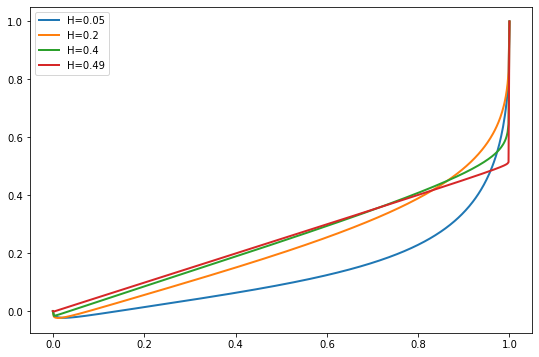}
   \includegraphics[width=0.47\linewidth]{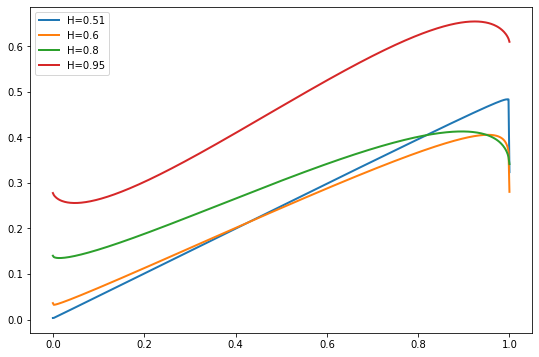}
   \caption{The numerical solution of the optimization problem \eqref{minpr} for various cases of Hurst index $H$ and $f(t)=t.$}
\end{figure}

\end{example}

\begin{example}
Here, we compare the approximate solution of equation \eqref{Feq} with the exact solution given by
\begin{equation}\nonumber
  x_2(t)=\begin{cases}
    8 \mathbbm{1}_{\lbrace x \leq 0.5 \rbrace} - \mathbbm{1}_{\lbrace 0.5 < x \leq  0.75 \rbrace} +4 \mathbbm{1}_{\lbrace 0.75  < x \leq 1 \rbrace} , & H \in (0,1/2),\\
    6-7t, & H \in (1/2,1).
  \end{cases}
\end{equation}
The right-hand side of \eqref{Feq} is computed by $f_2(t)=x_2(t)+C_2(H)\int_0^1 \varkappa_H(t,s)x_2(s)ds.$ The graphs of numerical solution and function $f$ are presented in Figure \ref{Fig3}. One can mention that approximate solution visually indistinguishable with the exact solution.
 \begin{figure}[htbp]
 \centering
 \label{Fig3}
  \includegraphics[width=0.45\linewidth]{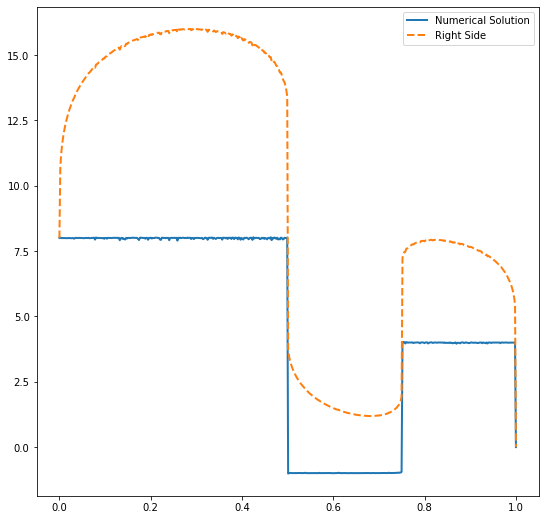}
  \includegraphics[width=0.45\linewidth]{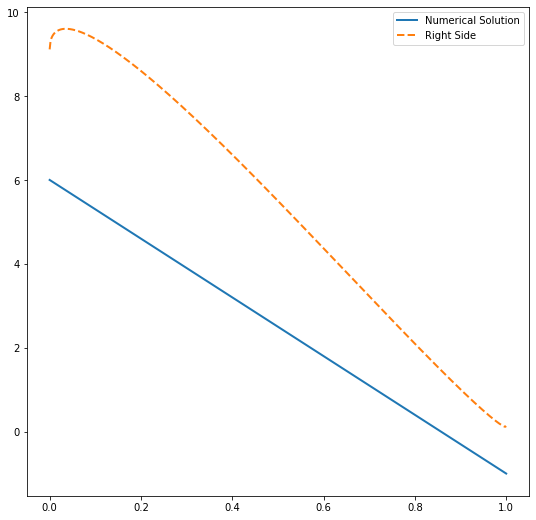}
  \caption{Numerical approximations of $x_2$ with $N=500$ for $H=0.25$ (left) and $H=0.75.$}
\end{figure}
\end{example}
In both examples, we obtain the solutions which can be negative on some interval and the right-hand side is non-negative simultaneously. This answers negative to the question about the existence of admissible optimal distribution $f_1=\alpha f$ and $f_2 (1-\alpha)f$ such that $0\leq \alpha(t)\leq 1,t\in [0,T].$

\begin{example}
In this example we study the sensitivity of $\max$ and $L_2[0,1]$-errors between the exact and approximate solutions with respect to the number $N.$ We provide the analysis for the quadratic solution $x_3(t)=t^2,t\in [0,1]$ and $H>1/2.$ In this case the right hand side of \eqref{Feq} equals
\begin{equation}\nonumber
  f_3(t)=t^2+C_3(H)\left[\frac{t^{2H+1}}{H(4H^2-1)}+\frac{(1-t)^{2H+1}}{2H+1}+\frac{t(1-t)^{2H}}{H}+\frac{t^2(1-t)^{2H-1}}{2H-1}\right],
\end{equation}
where $C_3(H)=C_2(H)B\left(H-\frac12,2-2H\right).$

The  maximum absolute error between the approximate and exact solution are presented in Table \ref{tb1}, the values of the $L_2-$norm of the error is listed in Table \ref{tb2}. From the computed values, we can vaguely estimate the error as $O(N^{-2}).$
\begin{table}[]
\centering
\caption{Maximum norm of the error scaled by $10^7$}
\label{tb1}
\begin{tabular}{l|rrrr}
$N$  & {$H = 0.51$} & {$H = 0.6$} & {$H = 0.8$} & {$H = 0.95$}    \\
\hline
\hline
25  & 149.26	& 946.55	& 1708.59	& 2697.23 \\
50 &40.19	& 242.03	& 414.26	& 647.79\\
100 & 10.89	& 62.43	& 102.18	& 158.78\\
200 &   2.95	&   16.12	&   25.40	&   39.31\\
300 & 1.37	& 7.30	& 11.27	& 17.41\\
500 & 0.52	& 2.68	& 4.05	& 6.25\\
\hline
\end{tabular}
\end{table}
\begin{table}[]
\centering
\caption{$L_2([0,1])$-norm of the error scaled by $10^7$}
\label{tb2}
\begin{tabular}{l|rrrr}
$N$  & {$H = 0.51$} & {$H = 0.6$} & {$H = 0.8$} & {$H = 0.95$}    \\
\hline
\hline
25  & 140.70	&892.52	&1630.11&2665.54  \\
50 & 38.21	&230.11	&396.76&	640.43 \\
100 & 10.42	&59.68	&98.07&	157.03 \\
200 &   2.83	&15.47&	24.41&	38.88 \\
300 & 1.32	&7.01&	10.84&	17.23 \\
500 & 0.50	&2.58&	3.90&	6.19 \\
\hline
\end{tabular}
\end{table}
\end{example}

\bibliographystyle{acm}
\bibliography{lit}

\end{document}